\documentclass[a4paper,12pt]{article}
\usepackage{geometry,latexsym,amssymb}
\usepackage[latin5]{inputenc}
\usepackage{enumerate}
\usepackage{enumitem}
\usepackage[usenames]{color}
\usepackage{graphicx}
\usepackage{enumerate}
\usepackage{rotating}
\usepackage{multirow}
\usepackage{cite}
\usepackage{amsthm}
\usepackage{tikz}
\usetikzlibrary{matrix,arrows}
\usepackage{amsmath,amscd}
\usepackage{calc}
\usepackage{ifthen}
\usepackage{tikz-cd} 
\usepackage{amssymb}
\usepackage{amssymb,amscd,amsthm, verbatim,amsmath,color,fancyhdr, mathrsfs}
\usepackage{graphicx}
\usepackage{turnstile}
\usepackage[colorinlistoftodos]{todonotes}
\usepackage{tikz-cd}

\newtheorem{thm}{Theorem}[section]
\newtheorem{prop}[thm]{Proposition}

\newtheorem{rem}[thm]{Remark}
\newtheorem{cor}[thm]{Corollary}

\newtheorem{defn}[thm]{Definition}

\begin{document}

\begin{center}
{\Large \textbf{A Note on the Soft Group Category}} \vspace*{0.5cm}
\end{center}

\vspace*{0.3cm}
\begin{center}
Nazmiye Alemdar$^{*,1}$, Hasan Arslan$^{*,2}$ \\
$^{*}${\small {\textit{Department of  Mathematics, Faculty of Science, Erciyes University, 38039, Kayseri, Turkey}}}\\
{\small {\textit{ $^{1}$nakari@erciyes.edu.tr}}}~~{\small {\textit{$^{2}$ hasanarslan@erciyes.edu.tr}}}\\

\end{center}

\begin{abstract}
The main purpose of this paper is to introduce the structure of soft group category. In this category, we determine some special objects and morphisms having a universal structure such as the final object and product. Therefore, the category of soft groups is a symmetric monodial category.
\end{abstract}


\\

\section{Introduction} 
\label{intro}
The real world is too complex for us to understand and interpret it. Therefore, simplified reality models of the real world are tried to be obtained. But these mathematical models are also very complex and it is very difficult to analyze them. Uncertainty in data makes classical methods unsuccessful when trying to shape problems in engineering, physics, computer science, economics, social sciences, health sciences, etc. Therefore, it is not entirely appropriate to use classical set theory based on exact cases when solving problems with such uncertainties. To deal with these problems, many mathematical theories such as fuzzy set theory, intuitionistic fuzzy set theory, indeterminate set theory, mathematical time theory and rough set theory have been defined.These theories are used as tools against uncertain situations. However, it has been seen that all these theories have their own problems. 

According to Molodtsov, these difficulties are most likely due to the inadequacy of the theory's parameterization tools. Molodtsov, freed from these troubles, put forward the soft set theory as a mathematical theory in 1999 in \cite{br4}. In his works, he successfully applied this new theory and its results to other fields such as probability theory, Perron integration, Riemann integration, research operation, game theory etc. In most of the games you need to design human behaviors or human models. There are many approaches to explain human behavior in game theory such as payment and selection functions. A selection function is a transformation that relates a set of strategies to a particular situation. Molodsov defined the \textit{soft function} as a mathematical tool that retains all the good sides of the selection function and eliminates the drawbacks of the pay function and the selection function. Tripathy et al. [7] also defined the basic definitions and concepts for soft sets in the decision-making process and provided the use of soft sets in the decision-making process based on game theory. 

Recently, many authors have worked on the algebraic structures of soft sets. Aktas and Cagman \cite{br1} introduced the soft group structure depending on definition of soft sets given by Molodsov and showed a way to construct algebraic structures based on concept of soft set.

As for category theory, the category theory was introduced by Samuel Elenberg and Saunders Maclane in the 1940s. The main purpose of category theory is to model and solve some problems in a simpler way by using objects and morphisms. 

Category theory is a comprehensive area of study in mathematics that examines in an abstract way the basic common language used to describe structures  occuring in different contexts. These developments in computer science plays very important role in programming studies, logic and authentication in computer science. 

In category theory, all information about objects is encoded with morphisms between them. In order to examine the internal structure of an object, not only the object itself but also the relations of this object with other objects in the category are considered. The characterization of the relations between any particular type of object and the rest of the universe in which it is located is called a universal structure and this situation is very common in category theory. For this reason, it should be investigated whether the category has special objects and morphisms, if any, in order to create the universal structure associated with the category.

If any category has a finite product and a final object and also both the product is a monoidal product and the final object is unit in this category, then this category is called a \textit{cartesian monoidal category} \cite{br2}. Any cartesian monoidal category is a symmetric monoidal category \cite{br3}. Open games can be considered as morphisms of a symmetric monoidal category with objects consisting of pairs of sets. However, morphisms in the symmetric monoidal category are also illustrations of Feynman diagrams in quantum field theory.

In this paper, we obtain the category of soft groups and investigate the structures of both special objects (final object and product) and morphisms (epimorphism, monomorphism)  in the category of soft groups as an analogue to Mac Lane's study \cite{br2}.

\section{Preliminaries}
\begin{defn}\label{}
Let $U$, $P(U)$ and $E$ be a  universal set, the power set of $U$ and a set of parameters, respectively. If $F: A \rightarrow P(U)$ is a function, then the pair $(F,A)$ is said to be a soft set over $U$, where $A \subseteq E$ \cite{br4}.  
\end{defn} 

\begin{defn}\label{}
Let $(F,A)$ and $(G,B)$ be any two soft sets over the same universal set $U$. If the following conditions are satisfied, then $(F,A)$ is called a soft subset of $(G,B)$ and we denote it by $(F,A) \sqsubseteq (G,B)$ \cite{br5}:
\begin{enumerate}
    \item $A \subseteq B$ \\
    \item $F(a)$ and $G(a)$ are two identical  predictions for every $a \in A$.
\end{enumerate}
\end{defn}
See \cite{br4}, \cite{br1}, \cite{br5} for more detailed information about soft sets.

\begin{defn}\label{}
Let $G$ be any group and let $(F,A)$ be soft set over $G$. If $F(a)$ is a subgroup of $G$ for each $a \in A$, then the pair $(F,A)$ is called a soft group over $G$ \cite{br1}.
\end{defn}

As a convention, throughout this paper, we denote by $(F,A)_G$ the soft group $(F,A)$ over the group $G$.

\begin{defn}\label{}
Let $(F,A)_G$ be a soft group.  If $F(x)=\{e\}$ for all $x \in A$, then $(F,A)_G$ is said to be a trivial or identity soft group, where $e$ stands for the identity element of $G$ \cite{br1}.
\end{defn}

\begin{defn}\label{}
Let $(F,A)_G$ be a soft group.  If $F(x)=G$ for all $x \in A$, then $(F,A)_G$ is called completely soft group \cite{br1}.
\end{defn}

\begin{defn}\label{}
Let $(F,A)_G$ and $(H,B)_K$ be two soft groups.  If there exists a group homomorphism $f: G \rightarrow K$ and a function $p:A \rightarrow B$ such that
$$\hat{f} \circ F=H \circ p$$
then the pair $(f,p)$ is said to be a soft group homomorphism from 
$(F,A)_G$ to $(H,B)_K$, where $f$ and $\hat{f}$ are identical on the power set $P(G)$ \cite{br6}.
\end{defn}
In other words, the pair $(f,p)$ is a soft group homomorphism if and only if the following diagram is commutative: 

\begin{equation}\label{sgh}
  \begin{tikzcd}
A \arrow{r}{F} \arrow[swap]{d}{p} & P(G) 
\arrow{d}{\hat{f}} \\%
B \arrow{r}{H} & P(K)
\end{tikzcd}
~~~~~~~~~~~~~~~~~~~~~~~~~~~
 \begin{tikzcd}
G \arrow[swap]{d}{f}\\%
K 
\end{tikzcd}  
\end{equation}

\begin{defn}\
Let the pair $(f,p)$ be a soft group homomorphism from 
$(F,A)_G$ to $(H,B)_K$. If $f$ is a group isomorphism and $p$ is a bijection then the soft groups $(F,A)_G$ and $(H,B)_K$ are called isomorphic and written $(F,A)_G \cong (H,B)_K$\cite{br6}.
\end{defn}

\begin{thm}\label{composition}
The composition of two soft group homomorphisms is a soft group homomorphism \cite{br6}.
\end{thm}

\begin{defn}[\cite{br1}]
Let $(F,A)_G$ and $(H,B)_K$ be two soft groups. \textit{Soft product} of the soft groups $(F,A)_G$ and $(H,B)_K$ is defined as 
$$U(x,y)=F(x)\times H(y)$$
for all $(x,y) \in A \times B$, and it is represented by
$$(F,A)_G \hat{\times} (H,B)_K=(U,A \times B)_{G \times K}.$$
\end{defn}
This concept can be generalized to three or finitely many soft groups in the following way:

\begin{defn}\label{gproduct}
Let $(F_1,A_1)_{G_1}, (F_2,A_2)_{G_2}, \cdots, (F_n,A_n)_{G_n}$ be soft groups. Soft product of these soft groups is defined as 
$$U(x_1,x_2,\cdots,x_n)=F_1(x_1) \times F_2(x_2) \cdots \times F_n(x_n)$$
for all $(x_1,x_2,\cdots, x_n) \in A_1 \times A_2 \cdots \times A_n$, and denoted by
$$(F_1,A_1)_{G_1} \hat{\times} (F_2,A_2)_{G_2} \hat{\times} \cdots \hat{\times} (F_n,A_n)_{G_n}=(U,A_1 \times A_2 \cdots A_n)_{G_1 \times G_2 \cdots \times G_n}.$$
\end{defn}

\section{Soft Group Category}
The objects of the soft group category are soft groups and the morphisms between these objects are soft group homomorphisms. The composition in this category is defined as the composition of soft group homomorphisms.
\begin{prop}
Soft groups and soft group homomorphisms between them forms a category. 
\end{prop}
\begin{proof}
The proof is clear from the equation (\ref{sgh}) and Theorem \ref{composition}.
\end{proof}
Note here that in this category for each object $(F,A)_G$, the unit morphism is defined as the soft group homomorphism $(1_G,1_A)$, where $1_G~:~G \mapsto G$ is the identity group homomorphism and $1_A~:~A\mapsto A$ is an identity map. Throughout this paper, this category is doneted by $SGp$.

We will give the definition of soft kernel which is not included in the literature and will be needed to prove that a monic morphism is one-to-one in the category of soft groups inspired by group theory.

\begin{defn}
Let $(F,A)_G$ and $(H,B)_K$ be any two soft groups. Let $(f,p): (F,A)_G \rightarrow (H,B)_K$ be a soft group homomorphism. Let $A'$ be the set consisting of $x \in A$ such that $F(x)=Kerf$. Having fixed 
\begin{equation} \label{kernel}
A'=\{x \in A : F(x)=Kerf \},
\end{equation}
we define $F'$ as the restriction function of $F$ to $A'$. Then the soft group $(F',A')_G$ is called the soft kernel of $(f,p)$.
\end{defn}
\begin{rem}
The soft kernel of a soft group homomorphism $(f,p): (F,A)_G \rightarrow (H,B)_K$ may not always be defined. One can construct a soft group homomorphism with the soft kernel $(f,p): (F,A)_G \rightarrow (H,B)_K$ by defining $F(x)=Kerf$ for some $x \in A$.
\end{rem}

\begin{thm}
Let $(F,A)_G$ and $(H,B)_K$ be any two soft groups. Assume that $(f,p): (F,A)_G \rightarrow (H,B)_K$ a soft group homomorphism such that $F(x)=Kerf$ for some $x \in A$. Then $f$ is injective if and only if $(F',A')_G$, which is the soft kernel of $(f,p)$, is the trivial soft subgroup.
\end{thm}

\begin{proof}
Let $f: G \rightarrow H$ be a monomorphism and let $F'$ be the restriction of the function $F$ to $A'$, where $A'$ is defined as in the equation (\ref{kernel}). Thus the pair $(F',A')_G$ is the trivial soft subgroup of $(F,A)_G$.
Conversely, let $(F',A')_G$ be the trivial soft subgroup, where $(F',A')_G$ is the soft kernel of $(f,p)$. We can write $F'(x)=Kerf=\{ e \}$ for all $x \in A'$. Therefore, $f$ is an injective group homomorphism.
\end{proof}

\begin{defn}\label{equality}
 Let $(f,p)$ and $(g,q)$ be two soft group homomorphisms. Then $(f,p)$ is equal to $(g,q)$ if and only if $f=g$ and $p=q$.
\end{defn}

\begin{thm}
Let $(f,p): (F,A)_G \rightarrow (H,B)_K$ be a soft group homomorphism such that $F(x)=Kerf$ for some $x \in A$. If $(f,p)$ is monic, then $f$ and $p$ are injective.
\end{thm}

\begin{proof}

\[ \begin{tikzcd}
A' \arrow{r}{F'} \arrow[swap]{d}{i_{A'}} & P(Kerf) 
\arrow{d}{\hat {i_{Kerf}}} \\%
A \arrow{r}{F} \arrow[swap]{d}{p}& P(G)
\arrow{d}{\hat{f}}\\%
B \arrow{r}{H} & P(K)
\end{tikzcd}
~~~~~~~~~~~~~~~~~~~~~~~~~~~
 \begin{tikzcd}
Kerf \arrow[swap]{d}{i_{Kerf}}\\%
G \arrow[swap]{d}{f}\\%
K 
\end{tikzcd}
\]
Suppose that $f$ is not injective. Thus we have $Kerf \neq \{e\}$. When $(f,p)\circ (i_{Kerf},i_{A'})=(f,p)\circ (f',i_{A'})$, we get  $(i_{Kerf},i_{A'}) \neq (f',i_{A'})$, where $f': Kerf \rightarrow G, f'(x)=e$. Hence the map $(f,p)$ is not monic. 

Now $p$ is not injective. Then we have two different $x,y \in A$ such that $p(x)=p(y)$.

\[ \begin{tikzcd}
A \arrow{r}{F} \arrow[swap]{d}{p_1,p_2} & P(G) 
\arrow{d}{\hat{1_G}} \\%
A \arrow{r}{F} \arrow[swap]{d}{p}& P(G)
\arrow{d}{\hat{f}}\\%
B \arrow{r}{H} & P(K)
\end{tikzcd}
~~~~~~~~~~~~~~~~~~~~~~~~~~~
 \begin{tikzcd}
G \arrow[swap]{d}{1_G}\\%
G \arrow[swap]{d}{f}\\%
K 
\end{tikzcd}
\]
If we define the maps $p_1,p_2:A \rightarrow A$ as
$$p_1(a)=
\begin{cases}
		a,  &  a \in A-\{x,y\} \\
		x,  &  a \notin A-\{x,y\}
\end{cases}
~~~\textrm{and}~~~~~
p_2(a)=
\begin{cases}
		a,  &  a \in A-\{x,y\} \\
		y,  &  a \notin A-\{x,y\}
\end{cases}
$$
then we have $p \circ p_1=p \circ p_2$.
Now we want to show that the map $(1_G,p_1): (F,A) \rightarrow (F,A)$ is a soft group homomorphism. In what follows, we need to prove the diagram 
\[ \begin{tikzcd}
A \arrow{r}{F} \arrow[swap]{d}{p_1,p_2} & P(G) 
\arrow{d}{\hat{1_G}} \\%
A \arrow{r}{F} & P(G)
\end{tikzcd}
~~~~~~~~\begin{tikzcd}
G\arrow[swap]{d}{1_G}\\%
G
\end{tikzcd}
\] 
is commutative. Suppose $a \in A-\{x,y\}$. Then we get 
$$(F \circ p_1)(a)=F(a)=\hat {1_G}(F(a))=(\hat {1_G} \circ F)(a).$$ 

Due the fact that $\hat{f}\circ F=H \circ p$ and $f$ is injective, we conclude  
$\hat{f}(F(x))=H(p(x))=H(p(y))=\hat{f}(F(y))$
and so 
\begin{equation}
    F(x)=F(y).
\end{equation}
Now we first take $a=x$. For any $x \in A$ we write $F(p_1(x))=F(x)=\hat {1_G}(F(x))$ and so $(F\circ p_1)(x)=(\hat {1_G}\circ F)(x)$. Secondly, if $a=y$, then we get $F(p_1(y))=F(x)=F(y)=\hat {1_G}(F(y))$ and again $(F\circ p_1)(y)=(\hat {1_G} \circ F)(y)$. Therefore $(1_G,p_1)$ is a soft group homomorphism. 

In the similar way, we can prove that the map $(1_G,p_2)$ is a soft group homomorphism. Eventually, we have
\begin{align*}
(f,p)\circ (1_G, p_1)&=(f\circ 1_G, p \circ p_1) \\
&=(f\circ 1_G, p \circ p_2) \\
&=(f,p)\circ (1_G, p_2).
\end{align*}
However, $(f,p)$ is not monic since $(1_G \circ p_1)\neq(1_G \circ p_2)$. This completes the proof.
\end{proof}

\begin{thm}
Let $(f,p): (F,A)_G \rightarrow (H,B)_K$ be a morphism in the soft group category $SGp$. If both $f$ and $p$ are injective, then the morphism $(f,p)$ is monic.
\end{thm}

\begin{proof}
Assume that $f$ and $p$ are injective. Let $(f_1,p_1),~(f_2,p_2) : (L,C)_M \rightarrow (F,A)_G$ be any two morphisms in $SGp$ such that $(f,p) \circ (f_1,p_1)=(f,p) \circ(f_2,p_2)$. Thus it can been easily seen that $(f \circ f_1,p \circ p_1) = (f \circ f_2,p \circ p_2)$. Then we get $f \circ f_1=f \circ f_2$ and $p \circ p_1 = p \circ p_2$. Since both $f$ and $p$ are one-to-one, we get $f_1 = f_2$ in $Gp$ and $p_1 = p_2$ in $Set$, where $Gp$ and $Set$ represent group category and set  category, respectively. Therefore the morphism $(f,p)$ is monic in $SGp$.
\end{proof}

\begin{thm}
If a soft group morphism $(f,p)$ in the $SGp$ is split monic, then both $f$ anf $p$ is injective.
\end{thm}

\begin{proof}
Let $(F,A)_G$ and $(H,B)_K$ be any two soft group and let $(f,p): (F,A)_G \rightarrow (H,B)_K$ be a soft split monic morphism. Since $(f,p)$ is a split monic morphism, then there exists $(g,q): (H,B)_K \rightarrow (F,A)_G$ such that $(g,q)\circ (f,p)=1_{(F,A)_G}=(1_G,1_A)$. Because of Definition \ref{equality}, we have $g \circ f=1_G$ and $q \circ p= 1_A$. It follows that $f$ is a split monic morphism in the group category $Gp$ and $p$ is a split monic morphism in the set category $Set$. Thus, both $f$ anf $p$ is injective.
\end{proof}

\begin{thm}
Let $(f,p): (F,A)_G \rightarrow (H,B)_K$ be a morphism in the soft group category $SGp$. If $f$ is an epimorphism in $Gp$ and $p$ is an onto function in $Set$, then the morphism $(f,p)$ is epic.
\end{thm}

\begin{proof}
Assume that $(g_1,q_1),~(g_2,q_2) : (H,B)_K \rightarrow (T,D)_N$ be any two morphisms in $SGp$ such that $(g_1,q_1)\circ (f,p) =(g_2,q_2)\circ (f,p)$. Thus we deduce that $(g_1 \circ f,q_1 \circ p) = (g_2 \circ f,q_2 \circ p)$. From this, we get $g_1 \circ f= g_2 \circ f$ and $q_1 \circ p = q_2 \circ p$. We obtain $g_1 = g_2$ in $Gp$ and $q_1 = q_2$ in $Set$ due to the facts that $f$ is an epimorphism and $q$ is onto. Hence, the morphism $(f,p)$ is epic in $SGp$.
\end{proof}

\section{Properties of the Soft Group Category}
In this section,  we will show that the SGp category is a symmetric monoidal category by proving the existence of the universal properties such as the final object and the product.

\begin{thm}
Let $\{e\}$ be the trivial group  and let $A=\{a\}$ be a singleton. Let $(F,\{a\})_{\{e\}}$ be a soft group such that $F: \{a\} \rightarrow P(\{e\}),~~F(a)=\{e\}$. Then $(F,\{a\})_{\{e\}}$ is a final object in the soft group category $SGp$.
\end{thm}

\begin{proof}
Let $(H,B)_K$ be an object in the soft group category $SGp$. Then the following diagram
\[ \begin{tikzcd}
B \arrow{r}{H} \arrow[swap]{d}{p} & P(K) 
\arrow{d}{\hat{f}} \\ %
\{a\} \arrow{r}{F} & P(\{e\})
\end{tikzcd}
\]
is commutative, where $p(b)=a$ for every $b \in B$ and
\[ \begin{tikzcd}
K \arrow[swap]{d}{f}\\ %
\{e\} 
\end{tikzcd}
\]
states the group homomorphism defined by $f(k)=e$ for each $k \in K$. Thus $(f,p)$ is the unique soft group homomorphism that can be defined from $(H,B)_K$ to $(F,\{a\})_{\{e\}} $. As a result, any soft group constructed by a parameter set with single-element and a group $\{e\} $ with one element is a final object in the soft group category $SGp$.
\end{proof}

\begin{prop}
Let $(F_1,A_1)_{G_1}$ and $(F_2,A_2)_{G_2}$ be soft groups and let $(U,A_1 \times A_2)_{G_1 \times G_2}=(F_1,A_1)_{G_1} \hat{\times} (F_2,A_2)_{G_2}$ be the soft product of them. For $i=1,2$ let $\Pi_i : A_1 \times A_2 \rightarrow A_i$ be $i$-th projection function on the parameter sets and let $p_i : G_1 \times G_2 \rightarrow G_i$ be $i$-th projection homomorphism on groups. Then the map $(p_i,\Pi_i): (U,A_1 \times A_2)_{G_1 \times G_2} \rightarrow (F_i,G_i)_{G_i}$ for each $i=1,2$ is a soft group homomorphism.
\end{prop}

\begin{proof}
To prove that $(p_i,\Pi_i)$ is a soft group homomorphism for each $i=1,2$ we need to show that the following diagram is commutative:
\[ \begin{tikzcd}
A_1 \times A_2 \arrow{r}{U} \arrow[swap]{d}{\Pi_i} & P(G_1 \times G_2) 
\arrow{d}{\hat{p_i}} \\%
A_i \arrow{r}{F_i} & P(G_i)
\end{tikzcd}
~~~~~~\begin{tikzcd}
G_1 \times G_2 \arrow[swap]{d}{p_i}\\%
G_i 
\end{tikzcd}
\] 

For this purpose, it is sufficient to show that the equality $\hat{p_i}\circ U= F_i \circ \Pi_i$ is satisfied. For any pair $(a_1,a_2)$, we get 
\begin{align*}
    (\hat{p_i}\circ U)(a_1,a_2)&=\hat{p_i}(F_1(a_1)\times F_2(a_2))\\
    &=F_i(a_i)\\
    &=F_i(\Pi_i(a_1,a_2))\\
    &=(F_i \circ \Pi_i)(a_1,a_2),
\end{align*}
and so the proof is completed.
\end{proof}

We can generalize the above proposition as follows:

\begin{cor}
Let $(F_1,A_1)_{G_1}, (F_2,A_2)_{G_2}, \cdots (F_n,A_n)_{G_n}$ be soft groups and let $(U,A_1 \times A_2 \times \cdots \times A_n)_{G_1 \times G_2\cdots \times G_n}=(F_1,A_1)_{G_1} \hat{\times} (F_2,A_2)_{G_2}\hat{\times} \cdots \hat{\times}(F_n,A_n)_{G_n}$ be the product of these soft groups in the sense of Definition \ref{gproduct}. For each $i=1,2, \cdots, n$, let $\Pi_i : A_1 \times \cdots \times A_n \rightarrow A_i$ be $i$-th projection function on the parameter sets and let $p_i : G_1 \times \cdots \times G_n \rightarrow G_i$ be $i$-th projection homomorphism on groups. Then the map $(p_i,\Pi_i): (U,A_1 \times \cdots \times A_n)_{G_1 \times \cdots \times G_2} \rightarrow (F_i,A_i)_{G_i}$ for each $i=1,2,\cdots, n$ is a soft group homomorphism.
\end{cor}

\begin{thm}
Let $(F_1,A_1)_{G_1}$ and $(F_2,A_2)_{G_2}$ be any two objects in $SGp$ and let $(U,A_1 \times A_2)_{G_1 \times G_2}=(F_1,A_1)_{G_1} \hat{\times} F_2,A_2)_{G_2}$ be the soft product of $(F_1,A_1)_{G_1}$ and $(F_2,A_2)_{G_2}$. Then the product of $(F_1,A_1)_{G_1}$ and $(F_2,A_2)_{G_2}$ in $SGp$ is no other than $((U,A_1 \times A_2)_{G_1 \times G_2}, (p_i,\Pi_i))$, where for each $i=1,2$,~ $\Pi_i : A_1 \times A_2 \rightarrow A_i$ is $i$-th projection function on the parameter sets and $p_i : G_1 \times G_2 \rightarrow G_i$ is $i$-th projection homomorphism on groups. 
\end{thm}

\begin{proof}

Assume that ${(H,B)}_K$ is an object and each $(g_i,q_i): (H,B)_K \rightarrow (F_i,A_i)_{G_i}$,$i=1,2$ is a  morphism in $SGp$.   
To prove that uniqueness of the product of $(F_1,A_1)_{G_1}$ and $(F_2,A_2)_{G_2}$, we need to show that there exists a unique soft group homomorphism $(\gamma,\theta): (H,B)_K \rightarrow (U,A_1 \times A_2)_{G_1 \times G_2}$, where we define $\theta=(q_1,q_2)$ and $\gamma=(g_1, g_2)$. For this reason, we will verify that the diagram below is commutative:

\[ \begin{tikzcd}
B \arrow{r}{H} \arrow[swap]{d}{\theta=(q_1,q_2)} & P(K)
\arrow{d}{\hat{\gamma}=\hat{g_1} \times \hat{g_2}} \\%
A_1 \times A_2 \arrow{r}{U} & P(G_1 \times G_2) 
\end{tikzcd}
~~~~~~\begin{tikzcd}
K \arrow[swap]{d}{\gamma}\\%
G_1 \times G_2
\end{tikzcd}
\]

For any element $b$ of $B$, we have
\begin{align*}
    (\hat{\gamma}\circ H)(b)&=\hat{\gamma}(H(b))\\
    &=\hat{g_1}(H(b))\times \hat{g_2}(H(b))\\
    &=F_1(q_1(b))\times F_2(q_2(b))\\
    &=U(q_1(b),q_2(b))\\
    &=(U\circ (q_1,q_2))(b)\\  
    &=(U\circ \theta)(b),
\end{align*}
so $(\gamma,\theta)$ is a soft group homomorphism. We illustrated the morphisms in the diagram below for the purpose of a better explanation of the subject.

\[
\begin{tikzpicture}
  \node (s) {$(H,B)_K$};
  \node (xy) [below=2 of s] {$(U,A_1 \times A_2)_{G_1 \times G_2}$};
  \node (x) [left=of xy] {$(F_1,A_1)_{G_1}$};
  \node (y) [right=of xy] {$(F_2,A_2)_{G_2}$};
  \draw[->] (s) to node [sloped, above] {$(g_2,q_2)$} (y);
  \draw[<-] (x) to node [sloped, above] {$(g_1,q_1)$} (s);
  \draw[->, dashed] (s) to node {$(\gamma,\theta)$} (xy);
  \draw[->] (xy) to node [below] {$(p_1, \Pi_1)$} (x);
  \draw[->] (xy) to node [below] {$(p_2, \Pi_2)$} (y);
\end{tikzpicture}
\]
Now we show  that for each $i=1,2$ the realtion $(p_i, \Pi) \circ (\gamma, \theta)=(g_i,q_i)$. For any $k\in K$, we have
\[
(p_i \circ \gamma)(k)=p_i ( \gamma(k))=p_i ( g_1(k), g_2(k))=g_i(k),
\]
and so $p_i \circ \gamma=g_i$. Similarly, we can prove the relation $\Pi_i \circ \theta=q_i$. 

Finally, let $(\gamma^{'},\theta^{'}): (H,B)_K \rightarrow (U,A_1 \times A_2)_{G_1 \times G_2}$ be another morphism satisfying the condition $(p_i, \Pi) \circ (\gamma^{'}, \theta^{'})=(g_i,q_i)$. Since 
$$(pi \circ \gamma^{'})(k)=g_i(k)=(pi \circ \gamma)(k)$$ 
for each $k \in K$, then we have $pi \circ \gamma^{'}=pi \circ \gamma$. We conclude that $ \gamma^{'}= \gamma$ due to the fact that each $p_i,~~i=1,2$ is a group monomorphism . In the similar manner, one can see $\theta^{'}=\theta$. Consequently, we obtain $(\gamma^{'},\theta^{'})=(\gamma,\theta)$, which means $(\gamma,\theta)$ is unique. Thus we complete the proof.
\end{proof}

\section{EXAMPLE}

In this section, we will give a soft group homomorphism with the soft kernel based on the structure of the hyperoctahedral group, which is a finite real reflection group. 

We assume that $[m,n]:=\{m, m+1, \cdots, n\}$ for any $m, n \in \mathbb{Z}$ such that $m \leq n$. Let $(W_n,R_n)$ be a Weyl group of type $B_n$, which is also called a \textit{hyperoctahedral group}, where $R_n=\{v_1, r_{1},\cdots, r_{n-1}\}$ is the canonical set of generators of $W_n$. Any element $w\in W_{n}$ acts as a signed permutation on the set $I_{n}=[ -n,n ] \backslash \{0\}$ such that $w(-i)=-w(i)$ for each $i \in I_{n}$. The group $W_n$ has a Dynkin diagram with respect to the set of generators $R_n=\{v_1,r_1,\cdots,r_{n-1}\}$ as follows:
\[
\begin{tikzpicture}
[
roundednode/.style={circle, draw=white!60, fill=white!5, very thick, minimum size=6mm},
roundnode/.style={circle, draw=black!50, fill=white!5, very thick, minimum size=6mm},
squarednode/.style={circle, draw=black!60, fill=white!5, very thick, text width=0.005mm, minimum size=6mm},
]
\node[roundnode, font=\tiny]      (maintopic){};
\node[roundednode]  [left=of maintopic]{$B_n$ :};
\node[roundnode]        (uppercircle)       [right=of maintopic]{};
\node[roundnode]      (rightsquare)       [right=of uppercircle]{};
\node[roundnode]        (lowercircle)       [right=of rightsquare]{};
\node[roundnode]        (zcircle)       [right=of lowercircle] {};
\node[node distance= 0.3cm, below = of maintopic]{$v_1$};
\node [label=$\frac{-1}{\sqrt{2}}$] (C) at (0.8,0) {};
\node[node distance= 0.3cm, below = of uppercircle]{$r_1$};
 \node[node distance= 0.3cm, below = of rightsquare]{$r_2$};
 \node[node distance= 0.3cm, below = of lowercircle]{$r_{n-2}$};
 \node[node distance= 0.3cm, below = of zcircle]{$r_{n-1}$};
\draw [-,double  distance=2pt,thick](maintopic) -- (uppercircle);
\draw[-] (uppercircle) -- (rightsquare);
\draw[densely dashed,-] (rightsquare)-- (lowercircle);
\draw[-] (lowercircle)-- (zcircle);
\draw[-](lowercircle)-- (zcircle);
\end{tikzpicture}\\
\]

The subgroup $W_{K}$ of $W_{n}$ generated by $K$ is said to be a \textit{standard parabolic subgroup} for any subset $K$ of $R_{n}$ and a subgroup of $W_{n}$ conjugate to $W_{K}$ for some $K \subset R_{n}$ is called  a \textit{parabolic subgroup}. Let $v_{i}:=r_{i-1}v_{i-1}r_{i-1}$ for each $2 \leq i \leq n$. It is well-known that $W_n=S_{n} \rtimes \mathcal{V}_{n} $, where $S_n$ is the symmetric group generated by $\{ r_1, \cdots, r_{n-1}\}$ and $\mathcal{V}_{n}$ is a normal subgroup of $W_n$ generated by reflections in $V_n=\{ v_1, \cdots, v_n \}$. That is, $W_{n}$ is a split group extension of $\mathcal{V}_{n}$ by $S_{n}$ and clearly the cardinality of the group $W_{n}$ is $2^{n} n!$. Note here that $r_i~~(i=1,\cdots,n-1)$ is identified with 
$$r_i(j)=
\begin{cases}
		i+1,  &  \textrm{if}~ j=i; \\
  	i,  &  \textrm{if}~ j = i+1;\\
		j,  &  \textrm{otherwise.}
\end{cases}
$$ and $v_i~~(i=1,\cdots,n)$ is defined by $$v_i(j)=
\begin{cases}
		j,  & \textrm{if} ~ j \neq i; \\
		-i,  & \textrm{if} ~ j=i.
\end{cases}
$$
The representation of $W_n$ is given by the following form (see \cite{br8}):
\begin{align*}
    W_n=&\langle  r_1,\cdots, r_{n-1}, v_1,\cdots, v_{n} : r_i^2=(r_ir_{i+1})^3=(r_ir_j)^2=e, |i-j|>1;\\ & v_{i}^2=(v_1r_1)^4=e, v_iv_j=v_jv_i, r_iv_ir_i=v_{i+1}, r_iv_j=v_jr_i,~j\neq i,i+1 \rangle
\end{align*}
where $e$ denotes the identity element of $W_n$.
For more detailed information about the Weyl group of type $B_n$, one can apply to \cite{br8}. 

A \textit{signed composition} of $n$ can be considered of as an expression of $n$ as an ordered sequence of nonzero integers. More precisely, a signed composition of $n$ is a finite sequence $A=(a_{1}, \cdots, a_{k})$ of nonzero integers satisfying 
 $\sum_{i=1}^{k}|a_{i}|=n$ \cite{br9}. Put $|A|=\sum_{i=1}^{k}|a_{i}|$. We will denote the set of all signed compositions of $n$ by $\mathcal{SC}(n)$. 

A \textit{bi-partition} of $n$ is a pair $\mu=({\mu}^{+}; {\mu}^{-})$ where ${\mu}^{+}$ and ${\mu}^{-}$ are partitions such that $|\mu|=|{\mu}^{+}|+|{\mu}^{-}|=n$. It is well-known from \cite{br9} that $\boldsymbol{\Lambda} : \mathcal{SC}(n) \rightarrow \mathcal{BP}(n),  \boldsymbol{\Lambda}(A)=(\boldsymbol{\Lambda}^{+}(A); \boldsymbol{\Lambda}^{-}(A))$ is a surjective map, where $\boldsymbol{\Lambda}^{+}(A)$ (resp. $\boldsymbol{\Lambda}^{-}(A)$) is rearrangement of the positive components (resp. absolute value of negative components) of $A$ in decreasing order. For $\mu=({\mu}^{+}; {\mu}^{-}) \in \mathcal{BP}(n)$, $\hat{\mu}:=\mu^{+} \sqcup -\mu^{-}$ is a unique signed composition obtained by concatenating the sequence of components of $\mu^{+}$ to that of $-\mu^{-}$ and $W_{\hat{\mu}}$ is a reflection subgroup of $W_n$ (see \cite{br9}).

We will denote by $\mathcal{BP}(n)$ the set of all bi-partitions of $n$. 
In \cite{br9}, Bonnaf\'e and Hohlweg assigned each signed composition of $n$ to a reflection subgroup of $W_{n}$ in the following way: The reflection subgroup $W_{A}$ of $W_{n}$ with respect to $A=(a_{1}, \cdots, a_{k})\in \mathcal{SC}(n)$ is generated by the reflection subset $R_{A}$, which is defined by
\begin{align*}
    R_{A}= & \{ r_{p}\in S_{n}~:~|a_{1}|+ \cdots + |a_{i-1}|+1 \leq p < |a_{1}|+ \cdots + |a_{i}|\} \\
           & \cup \{ {v_{|a_{1}|+ \cdots + |a_{j-1}|+1}} \in V_{n} \} ~|~ a_{j} > 0 \} \subset R_{n}^{'}
\end{align*}
where $R'_{n}=\{r_{1}\cdots r_{n-1}, v_{1}, v_{2}, \cdots, v_{n}\}.$

Let $\mathcal{P}(W_n)$ denote the power set of $W_n$. If we pick the parameter set as $\mathcal{BP}(n)$ and accordingly define the map $G: \mathcal{BP}(n) \rightarrow \mathcal{P}(W_n),~~G(\mu)=W_{\hat{\mu}}$, then the pair $(G, \mathcal{BP}(n))_{W_n}$ is a soft group.

Now we give another soft group example. If we take $\mathcal{SC}(n)$ as the parameter set and define the map $F$ as $F(A)=W_{\hat{\mu}}$, where we can write $\hat{\mu}:=\boldsymbol{\Lambda}^{+}(A) \sqcup -\boldsymbol{\Lambda}^{-}(A) \in \mathcal{SC}(n)$ for every $A \in \mathcal{SC}(n)$, then the pair $(F, \mathcal{SC}(n))_{W_n}$ is a soft group due to the fact that $W_{\hat{\mu}}$ is a subgroup of $W_n$.

The pair $(f, \boldsymbol{\Lambda}): (F, \mathcal{SC}(n))_{W_n} \rightarrow (G, \mathcal{BP}(n))_{W_n}$ is a soft group homomorphism with the soft kernel. Therefore, we say that the following diagram is commutative: 
\[ \begin{tikzcd}
\mathcal{SC}(n) \arrow{r}{F} \arrow[swap]{d}{\boldsymbol{\Lambda}} & \mathcal{P}(W_n) 
\arrow{d}{}{\hat{f}}\\%
\mathcal{BP}(n) \arrow{r}{G}  & \mathcal{P}(W_n)
\end{tikzcd}
\]
where $f~:~W_n \mapsto W_n$ is the trivial isomorphism and $\hat{f}$ is the identity function on $\mathcal{P}(W_n)$. In fact, let $A \in \mathcal{SC}(n)$ and let $\boldsymbol{\Lambda}(A):=\mu$. Thus we can write $\hat{\mu}=\boldsymbol{\Lambda}^{+}(A) \sqcup -\boldsymbol{\Lambda}^{-}(A)$ and so it is clear that $ \boldsymbol{\Lambda}(\hat{\mu})=\mu$ for every $\mu \in \mathcal{BP}(n)$. Therefore, for every $A \in \mathcal{SC}(n)$, we get
\begin{align*}
    (G\circ \boldsymbol{\Lambda})(A)&=W_{\hat{\mu}}\\
    &=\hat{f}(W_{\hat{\mu}})\\
    &=\hat{f} \circ F(A).
\end{align*}
Hence the above diagram is commutative. The soft kernel of this soft group homomorphism equals to the set $\{(-1,-1,\cdots,-1)\}$, since $F((-1,-1,\cdots,-1))=\{e\}=Kerf$ is the trivial subgroup of $W_n$.


\vspace*{0.3cm}


\begin{thebibliography}{99}

\bibitem{br4} Molodtsov, Dmitri A: Soft set theory-First results, Computers Mathematics With Applications, 37 (1999)  19-31.

\bibitem{br1} H. Aktas, N. Cagman: Soft sets and soft groups, Inf. Sci., 177 (2007)  2726-2735 . 

\bibitem{br2} S. Mac Lane: Categories for the Working Mathematician, Graduate Text in Mathematics, Volume 5 Springer Verlag, Newyork, (1971). 

\bibitem{br3} J. Baez and M. Stay: Physics, topology, logic and computation: a Rosetta Stone, in New Structures for Physics, ed. Bob Coecke, Lecture Notes in Physics vol. 813, Springer, Berlin,  (2011) 95-174. 

\bibitem{br5} P.K Maji, R. Bismas, A.R. Roy: Soft set theory, Comput. Math. Appl. 45(2003) 555-562.

\bibitem{br6} A. Sezgin ,  A.O. Atagun: Soft Groups and Normalistic Soft Groups, Comput. Math. Appl, 62 (2)  (2011) 1457-1467.  
\bibitem{br8} J. E. Humphreys, Reflection Groups and Coxeter Groups, Cambridge University Press Volume 29, 1990.

\bibitem{br9} C. Bonnafe, C. Hohlweg: Generalized descent algebra and construction of irreducible characters of hyperoctahedral groups, Ann. Inst. Fourier (Grenoble), 56(1) (2006) 131-181.
\end{thebibliography}
\end{document}